\documentclass[11pt, letterpaper]{article}
\usepackage[left = 1in, right = 1in, top = 1in, bottom = 1in]{geometry}
\usepackage{amsmath, amssymb, graphicx, titlesec, textcomp, mathtools, amsthm}
\usepackage{xypic,tikz-cd}
\usepackage{stmaryrd}
\usepackage[all]{xy}
\usepackage{faktor}
\usepackage{tikz}	
\usetikzlibrary{matrix}
\usetikzlibrary{patterns}
\usetikzlibrary{matrix}
\usetikzlibrary{positioning}
\usetikzlibrary{decorations.pathmorphing}
\newcommand{\set}[1]{\{#1\}}
\newcommand{\Z}{\mathbb{Z}}

\newcommand{\C}{\mathbb{C}}

\newcommand{\M}{\mathcal{M}_{0,n}}
\newcommand{\Mbar}{\overline{\mathcal{M}}_{0,T}}
\renewcommand{\P}{\mathbb{P}^1(\mathbb{C})}

\newcommand{\defas}{\mathrel{\mathop:}=}
\newtheorem{defn}{Definition}
\newtheorem{thm}{Theorem}

\newtheorem{lem}{Lemma}

\newtheorem{remark}{Remark}
\numberwithin{defn}{section}
\numberwithin{thm}{section}
\numberwithin{prop}{section}
\numberwithin{lem}{section}
\numberwithin{conj}{section}
\numberwithin{remark}{section}

\title{Boundary Expression for Chern Classes of the Hodge Bundle on Spaces of Cyclic Covers}
\author{Bryson Owens and Seamus Somerstep}
\date{\vspace{-5ex}}
\begin{document}
\maketitle
\begin{abstract}
  We compute an explicit formula for the first Chern class of the Hodge Bundle over the space of admissible cyclic $\faktor{\Z}{3\Z}$ covers of $n$-pointed rational stable curves as a linear combination of boundary strata. We then apply this formula to give a recursive formula for calculating certain Hodge integrals containing $\lambda_1$. We also consider covers with a $\faktor{\Z}{2\Z}$ action for which we compute $\lambda_2$ as a linear combination of codimension two boundary strata.
\end{abstract}
\thispagestyle{empty}
\section{Introduction}
This paper studies the intersection theory of moduli spaces of cyclic admissible covers. We extend a result about the first Chern class of the Hodge bundle over spaces of degree two admissible covers of nodal stable curves, which first appeared in an unpublished undergraduate honors thesis \cite{pete} and two years later in \cite{champs}, to the case of cyclic degree three covers. We begin with the Deligne-Mumford compactification of the moduli space of $T$-pointed curves of genus 0, $\overline{\mathcal{M}}_{0,T}$. Let $\omega = (123), \overline{\omega} = (132) \in \faktor{\Z}{3\Z}$. We then consider the moduli space $Adm_{g \xrightarrow{3} 0(n|m)}$ of admissible genus $g$ covers of $(n+m)$-pointed genus 0 curves which ramify over the marked points of the covered curve, such that $n$ branch points have monodromy $\omega$ and $m$ branch points have monodromy $\overline{\omega}$. For notational simplicity, we denote this space as $Adm_{g(n|m)}$, where the genus of the curve being covered is always 0 and the degree of the cover is understood in context (the degree shall be 3 throughout this paper except in $\mathsection$ 4, where we consider covers of degree 2). We denote by $\lambda_i$ the $i$-th Chern class of the Hodge Bundle over $Adm_{g(n|m)}$ and by $D_i^j$ the sum of all irreducible codimension one boundary strata parameterizing nodal covers with $i$ branch points of monodromy $\omega$ and $j$ branch points with monodromy $\overline{\omega}$ on one component. The branch morphism $br: Adm_{g(n|m)} \rightarrow \Mbar$ is a bijection on the points of the moduli spaces, but has degree due to the fact that every admissible cover has a cyclic automorphism group of order three. From \cite{keel}, we know that the Chow ring $A^*(\overline{\mathcal{M}}_{0,T})$ is generated by its boundary strata. From this it follows that $\lambda_1$, a codimension one tautological class of $Adm_{g(n|m)}$, can be expressed as the pullback via the branch morphism of a linear combination of boundary divisors of $\overline{\mathcal{M}}_{0,T}$. Our main result in this paper gives an explicit formula for computing the coefficients of this linear combination.

\begin{thm}
$\lambda_1$ can be expressed as $3\pi^*(\sum \alpha_i^jD_i^j)$ over all symmetrized boundary divisors $D_i^j$ of $\overline{\mathcal{M}}_{0,T}$,
$$\text{where } \alpha_i^j = \begin{cases}
\frac{2(i+j)(T-i-j)}{27(T - 1)}&i - j \equiv 0 \; mod \; 3\\
\frac{2(i+j - 1)(T-i-j-1)}{27(T - 1)}& i - j \equiv \pm 1 \; mod \; 3
\end{cases}$$
\end{thm}
We prove this theorem by computing for any curve $\gamma$ in $\overline{\mathcal{M}}_{0,T}$
$\pi_*(\lambda_1) \cdot \gamma$ and $\sum \alpha_i^jD_i^j \cdot \gamma$ using the formula described in Theorem 1.1 and verify these two expressions are always equal. As an application of this formula we give the following recursive method for calculating certain Hodge integrals containing $\lambda_1$.
\begin{thm}
The family of Hodge integrals $\int_{Adm_{g(n|m)}} \hspace{-0.1cm} \lambda_{1}^{n+m-3}$ is given by the recursive formula:
$$\int_{Adm_{g(n|m)}} \hspace{-1.3cm} \lambda_{1}^{n+m-3} = 3 \sum_{i=0}^{n} \sum_{j=0}^{m} \frac{2(i+j-1)(T-i-j-1)}{9(T-1)} {n+m-3 \choose i+j-2}{n \choose i}{m \choose j}\int_{Adm_{(i+1|j)}} \hspace{-1.3cm} \lambda_{1}^{i+j-2}\int_{Adm_{(n-i|m-j+1)}} \hspace{-2.3cm} \lambda_{1}^{n+m-i-j-2} $$
\end{thm}
This theorem follows naturally from Theorem 1.1 as well as how the Hodge bundle, and consequently its Chern classes, split.

We also study the second Chern class, $\lambda_2$, in the case of degree two admissible cover.  Using Mumford's relations and the expression for $\lambda_1$ in \cite{champs} we prove the following theorem.
\begin{thm}
Let $\Delta_{i_1,i_2,i_3}$ denote the codimension 2 symmeterized stratum in $Adm_g$ with $i_1$ branch points on the left component, $i_2$ branch points on the middle component, and $i_3$ branch points on the right component. Then $\lambda_2 = \sum \alpha_{i_1,i_2,i_3} \Delta_{i_1,i_2,i_3}$, where
$$\ \alpha_{i_1,i_2,i_3} = \begin{cases}
\frac{i_1i_2i_3(2i_1i_2+2i_1i_3+2i_2i_3-i_1-2i_2-i_3)}{32(i_1+i_2+i_3-1)(i_1+i_2-1)(i_2+i_3-1)}&i_1,i_2,i_3 \equiv 0 \; mod \; 2\\
 \frac{(i_1-1)(i_2)(i_3-1)((i_2+i_3-1)(i_1+i_2)+(i_1+i_2-1)(i_2+i_3))}{32(i_1+i_2+i_3-1)(i_1+i_2)(i_2+i_3)}&i_1, i_3 \equiv 1, \ \ i_2 \equiv 0 \; mod \; 2 \\
 \frac{(i_1-1)(i_2+i_3-1)(i_2+1)(i_3)(i_1+i_2-1)+(i_3)(i_1+i_2)(i_2-1)(i_1-1)(i_2+i_3)}{32(i_1+i_2+i_3-1)(i_2+i_3)(i_1+i_2-1)}&i_1, i_2 \equiv 1, \ \ i_3 \equiv 0 \; mod \; 2 \\
 
\end{cases}$$  
\end{thm}

\textbf{Acknowledgements}

The authors thank Dr. Renzo Cavalieri for suggesting the problem and for his helpful guidance.  They also thank Adam Afandi for his help in teaching us the intersection theory of $\Mbar$
 
\section{Preliminaries}
\subsection{The Moduli Space of Pointed Stable Rational Curves $\M$}
We denote by $\mathcal{M}_{0,n}$ the moduli space of $n$ marked points on $\P$, up to the action of $\mathbb{P}$GL(2,$\C$). That is, each point on $\mathcal{M}_{0,n}$ corresponds to an isomorphism class of $n$ distinct marked points on $\P$. The theory of M\"{o}bius transformations tells us that there exists a unique automorphism on $\P$ that sends any 3-tuple of points to any other 3-tuple of points. That is, given an n-tuple $p_1$, ... $p_n$, there exists a unique M\"{o}bius transformation $\Phi: \P \rightarrow \P$ such that $\Phi(p_1) = 0, \Phi(p_2) = 1, \Phi(p_3) = \infty$ (these numbers are chosen simply by convention and could be chosen to be any other three points in $\P$) and the other marked points are determined uniquely as the images by $\Phi$ of $p_4, ..., p_n$. Since any $\P$ marked by some $n$-tuple is isomorphic to a $\P$ marked with an $n$-tuple whose first three coordinates are $0, 1, \infty$ by some M\"obius Transformation, we may pick the latter copy of $\P$ to be the representative for the isomorphism class. Therefore, since each isomorphism class has equivalent first three points, we may parametrize them by their remaining $n-3$ points. Therefore, $\mathcal{M}_{0,n}$ is $n-3$ complex dimensional. For example, $\mathcal{M}_{0,3}$ is a single point since all copies of $\P$ marked with three points are isomorphic to $\P$ marked with $0, 1, \infty$. While $\M$ is not a compact topological space, it admits a compactification known as the Deligne-Mumford compactification, denoted $\Mbar$, which parameterizes nodal stable curves. 
\begin{defn}
A nodal stable curve is a tree of projective lines, which have the following properties:
\begin{itemize}
  \item Components of the tree are copies of $\P$ connected at nodes.
  \item There are no closed circuits.
  \item Each component has at least three special points (marked points or nodes).
\end{itemize}
\end{defn}
The complement of $\M$ in $\Mbar$, called the boundary of $\Mbar$, is the set of points parameterizing marked stable nodal curves. This boundary is stratified by topological type, with each stratum being indexed by a dual graph, which we now define.

\begin{defn}
Given a nodal stable curve $C$ \textbf{dual graph} is a connected graph with the following properties:
\begin{itemize}
  \item Each vertex corresponds to a copy of $\P$.
  \item Half-edges connected to a vertex correspond to marked points on the component corresponding to that vertex.
  \item Edges between vertices correspond to nodes connecting the components corresponding to the vertices.
\end{itemize}
\end{defn}
A given dual graph only represents how each component is connected to the others and which marked points are on which components. Therefore, dual graphs represent strata (loci) of points rather than a single boundary point.

We use this combinatorial representation for calculations in this paper. As an example, consider the nodal curve consisting of two copies of $\P$ each containing two marked points. This curve is represented by the dual graph in Figure \ref{boundarym4} and parametrizes a point in the boundary of $\overline{\mathcal{M}}_{0,4}$ (the other two points in the boundary of $\overline{\mathcal{M}}_{0,T}$ are parametrized by dual graphs with the same structure, with the marked points relabeled):

\begin{figure}[h]
\begin{center}
\begin{tikzpicture}

\fill (.3,0) circle (0.10);
\draw[very thick] (.3,0) -- (-0.2,0.5); \node at (-0.5,0.5) {$p_1$};
\draw[very thick] (.3,0) -- (-0.2,-0.5); \node at (-0.5,-0.5) {$p_2$};
\draw[very thick] (.3,0) -- (1.5,0);

\fill (1.5,0) circle (0.10);
\draw[very thick] (1.5,0) -- (2.0,0.5); \node at (2.3,0.5) {$p_3$};
\draw[very thick] (1.5,0) -- (2.0,-0.5); \node at (2.3,-0.5) {$p_4$};


\end{tikzpicture}\end{center}
\caption{Boundary Point in $\overline{\mathcal{M}}_{0,4}$}
\label{boundarym4}
\end{figure}
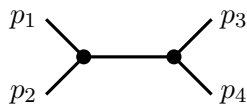

\begin{defn}
Consider $\overline{\mathcal{M}}_{0,n+m}$ and the action of $S_n \times S_m$ where the element in $S_n$ permutes the first $n$ points and the element in $S_m$ permutes the last $m$ points. Then, a \textbf{symmetrized stratum} is the orbit of a stratum via this action.
\end{defn}
For example, consider Figure $\ref{boundarym4}$. The other two boundary points in $\overline{\mathcal{M}}_{0,4}$ have dual graphs with the same structure, but with the marked points relabeled as follows:\\
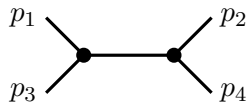
\begin{figure}[h]
\begin{center}
\begin{tikzpicture}

\fill (.3,0) circle (0.10);
\draw[very thick] (.3,0) -- (-0.2,0.5); \node at (-0.5,0.5) {$p_1$};
\draw[very thick] (.3,0) -- (-0.2,-0.5); \node at (-0.5,-0.5) {$p_3$};
\draw[very thick] (.3,0) -- (1.5,0);

\fill (1.5,0) circle (0.10);
\draw[very thick] (1.5,0) -- (2.0,0.5); \node at (2.3,0.5) {$p_2$};
\draw[very thick] (1.5,0) -- (2.0,-0.5); \node at (2.3,-0.5) {$p_4$};


\end{tikzpicture}\end{center}
\end{figure}\\
or\\
\begin{figure}[h]
\begin{center}
\begin{tikzpicture}

\fill (.3,0) circle (0.10);
\draw[very thick] (.3,0) -- (-0.2,0.5); \node at (-0.5,0.5) {$p_1$};
\draw[very thick] (.3,0) -- (-0.2,-0.5); \node at (-0.5,-0.5) {$p_4$};
\draw[very thick] (.3,0) -- (1.5,0);

\fill (1.5,0) circle (0.10);
\draw[very thick] (1.5,0) -- (2.0,0.5); \node at (2.3,0.5) {$p_3$};
\draw[very thick] (1.5,0) -- (2.0,-0.5); \node at (2.3,-0.5) {$p_2$};


\end{tikzpicture}\end{center}
\end{figure}\\
Then, there is a single symmetrized strata which is the union of the three boundary points above. We denote this symmetrized stratum as an unlabeled copy of the dual graphs of the boundary points in the symmetrized stratum. In the example of boundary points of $\overline{\mathcal{M}}_{0,4}$ this is shown in Figure \ref{symmbdry4}\\
\begin{figure}[h]
\begin{center}
\begin{tikzpicture}

\fill (.3,0) circle (0.10);
\draw[very thick] (.3,0) -- (-0.2,0.5);
\draw[very thick] (.3,0) -- (-0.2,-0.5);
\draw[very thick] (.3,0) -- (1.5,0);

\fill (1.5,0) circle (0.10);
\draw[very thick] (1.5,0) -- (2.0,0.5);
\draw[very thick] (1.5,0) -- (2.0,-0.5);


\end{tikzpicture}\end{center}
\caption{Symmetrized Stratum in $\overline{\mathcal{M}}_{0,4}$}
\label{symmbdry4}
\end{figure}
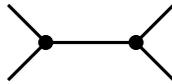

The codimension 1 boundary strata of $\Mbar$ are called \textbf{boundary divisors}. These are especially important for us as $\pi_*(\lambda_1)$ is expressible as a linear combination of boundary divisors of $\Mbar$. A nice combinatorial representation of the codimension of boundary strata is the number of edges in the dual graph of the stratum. For example, a boundary divisor has one edge connecting two vertices, each with a certain number of half-edges connected to it.

\subsection{The Chow Ring $A^*(\Mbar)$}
The Chow Ring of $\Mbar$ is a ring whose elements are equivalence classes of subvarieties of $\Mbar$, graded by codimension. $A^i(\Mbar)$ is then the section of the ring corresponding to codimension $i$ subvarieties of $\Mbar$. From Keel \cite{keel}, we know that $A^*(\Mbar)$ is generated by the boundary strata of $\Mbar$. The binary operations on $A^*(\Mbar)$ are formal addition and a multiplication, which corresponds to intersection of strata. Naively, we may expect this multiplication to simply be set theoretic intersection, however, as we will see, this is not always the case. To see why, note that since $A^*(\Mbar)$ is a graded ring, the multiplication must be a map
$$\cdot : A^i(\Mbar) \times A^j(\Mbar) \rightarrow A^{i + j}(\Mbar)$$
That is, $codim(X \cdot Y) = codim(X) + codim(Y)$. For an easy example of where set theoretic intersections fail to satisfy this codimension requirement, consider a self intersection. Obviously, $X \cap X = X$, so the codimensions can only follow the above condition if $X = \Mbar$. Therefore, we must consider two types of intersections: transverse and non-transverse, where multiplication on the Chow Ring is set theoretic intersection if and only if the intersection is transverse. If the intersection is non-transverse, which we define in $\sec 2.4.1$, then we must algebraically deform one of the subvarieties so that the intersection respects codimension. In order to actually compute these non-transverse intersections, we require the use of $\psi$-classes.

\subsection{Chern Classes}
Let $E$ be a complex rank $r$ vector bundle. The $i$-th Chern class, $c_i(E)$ is a codimension $i$ characteristic class of $E$, which has the following properties \cite{milnor} which we use throughout this paper:
\begin{itemize}
    \item $c_0(E) = 1$
    \item $c_i(E) = 0$ for $i > r$
    \item $c_1(E \oplus F) = c_1(E) + c_1(F)$
\end{itemize}

\subsection{$\psi$-Classes}
In order to fully understand the intersection theory of $\Mbar$ we must first introudce the ideas of $\psi$-classes. First, note that $\Mbar$ admits sections in its universal family. We have a sheaf on the universal family whose stalks are differential 1-forms on the curve parametrized. This is called the \textit{relative dualizing sheaf}, and is denoted $\omega_\pi$. This gives a means to define the i-th $\psi$-class 

\begin{defn}
The i-th $\psi$-class, denoted $\psi_i$ is the first Chern class of the restriction of $\omega_\pi$ to the i-th section of $\Mbar$ 
$$\psi_i = c_1(s_i^*(\omega_{\pi}))$$
\end{defn}

There is an important lemma that is helpful in computing $\psi$ classes on $\Mbar$. It states that $\psi_i$ on $\Mbar$ is equal to the sum of all boundary divisors where the i-th point is fixed on the left twig and two other marked points are fixed on the right twig. For our work the two important computations are that $\psi_i$ on $\overline{\mathcal{M}}_{0,3}$ is zero (there is no way to distribute three points to obtain a stable boundary divisor) and $\psi_i$ on $\overline{\mathcal{M}}_{0,4}$ is the class of a point (there is only one way to distribute the last point so the dual graph is stable).

\subsection{Intersections of Strata}
Take two strata in $\Mbar$ with dual graphs $\Gamma_1, \Gamma_2$. In order to calculate the intersection of the two strata, we must find the minimal refinement of the graphs $\Gamma_1, \Gamma_2$. A refinement is a graph $\Gamma$ such that both $\Gamma_1, \Gamma_2$ can be obtained by contracting edges. Color the edges which are \textit{not} contracted to obtain $\Gamma_1$ red and the edges which are \textit{not} contracted to obtain $\Gamma_2$ blue. The refinement is said to be minimal if all edges are colored red, blue, or both.

This method of coloring edges which are not contracted to obtain the original dual graphs from the refinement give a simple test for whether an intersection is transverse or non-transverse. The intersection is transverse if no edge in the minimal refinement is colored both red and blue, and the minimal refinement $\Gamma$ is dual to the product of the strata to which $\Gamma_1, \Gamma_2$ are dual.

Furthermore, an intersection is non-transverse if there exists an edge in $\Gamma$ which is colored both red and blue. This bicolored edge is called a common edge and corresponds to curves parametrized by the two strata having a common node. We can now use $\psi$-classes to compute non-transverse intersections.

\subsubsection{Non-Transverse Intersections}
Let $S$ be a boundary divisor of $\overline{\mathcal{M}}_{0,n}, T_1$ be the set of marked points on one component, and $T_2$ the set of marked points on the other component. Then, $S \cong \overline{\mathcal{M}}_{0,T_1 \cup \set{\cdot}} \times \overline{\mathcal{M}}_{0,T_2 \cup \set{\star}}$ where $\cdot, \star$ are the two marked points which, when glued together, form the node of the curve parametrized by $S$. Then, there exist projections as shown in Figure \ref{split}.

\begin{figure}[h]
\[ \begin{tikzcd}
& S \cong \overline{\mathcal{M}}_{0,T_1 \cup \set{\cdot}} \times \overline{\mathcal{M}}_{0,T_2 \cup \set{\star}} \arrow[swap]{dl}{\rho_1} \arrow{dr}{\rho_2} \arrow[hookrightarrow]{r}{i}&\Mbar \\
\overline{\mathcal{M}}_{0,T_1 \cup \set{\cdot}}& &\overline{\mathcal{M}}_{0,T_2 \cup \set{\star}}
\end{tikzcd} \]
\caption{Splitting a boundary divisor into irreducible components}
\label{split}
\end{figure}
We can then define the notion of adding $-\psi$-classes at half-edges $\cdot$ and $\star$ as follows:
$$-\psi_{\cdot} - \psi_{\star} \defas i_*(\rho_1^*(-\psi_{\cdot}) + \rho_2^*(-\psi_{\star}))$$
In order to compute a non-transverse intersection, we use this idea of adding $-\psi$-classes at the common edge. That is, say two strata $S_1, S_2$ intersect non-transversally and let the refinement of the two graphs dual to $S_1, S_2$ be $\Gamma$. Let the bicolored edges of $\Gamma$ break into half-edges $\cdot_i, \star_i$ by projections $\rho_1, \rho_2$. Then, 
$$S_1 \cdot S_2 = \prod -\psi_{\cdot_i} - \psi_{\star_i}$$
over each common edge, supported on the stratum to which $\Gamma$ is dual.

\subsection{The Space of Cyclic Admissible Covers}
\begin{defn}
A degree d \textbf{cyclic admissible cover} is a curve $C$ along with a map $\pi: C \rightarrow X$ of degree d which satisfies the following: 
\begin{itemize}
  \item $C$ has an action by the cyclic group of order d, and $\pi$ is the quotient map.
  \item $\pi$ is \`{e}tale everywhere except at a finite set of points, called the branch locus. That is, except at branch points (points in the branch locus), the fiber consists of $d$ points. The fiber of a branch point will always consist of fewer than $d$ points.
  \item Each branch point has a mondromy representation in the cyclic group of degree $d$ .
  \item $C$ and $X$ are nodal curves, and the image of a node in $C$ by $\pi$ is a node in $X$.
  \item Over a node, locally in analytic coordinates, $X, C, \pi$ are described as follows, for some positive integer r not larger than d:
  \begin{itemize}
    \item $C: c_1c_2=a$
    \item $X: x_1x_2=a^r$
    \item $x_1=c_1^r, x_2=c_2^r$
  \end{itemize}
\end{itemize}
\end{defn}

For our main results we wish to look at the moduli space of admissible covers of a marked rational curve, which ramify over the marked points. For this paper, we are focusing on the case of degree two and degree three covers. In the degree two case, the group $\faktor{\Z}{2\Z}$ acts on $C$, meaning each branch point is fully ramified. In the degree three case we have a $\faktor{\Z}{3\Z}$ action on the admissible covers. Note that $\faktor{\Z}{3\Z} \cong \langle(123)\rangle \leq S_3$. This is why each branch point has monodromy $(123), \text{ or } (132)$, which we denote as $\omega$ and $\overline{\omega}$, respectively. We denote the moduli space parameterizing admissible covers with $n$ $\omega$ points and $m$ $ \overline{\omega}$ points as $Adm_{g(n|m)}$, which has dimension $n + m - 3$ and parametrizes curves of genus $n + m - 2$.

Finally, note that there exists a bijection from $Adm_{g(n|m)} \rightarrow \overline{\mathcal{M}}_{0,n+m}$ called the branch morphism. The branch morphism has degree $\frac{1}{d}$ since there are $d$ automorphisms of any admissible cover.
\begin{remark}
Let $\Delta$ be a boundary divisor in $Adm_{n|m}$. Then, $\Delta \cong d (Adm_{n_1|m_1} \times Adm_{n_2|m_2})$ for some $n_1, n_2, m_1, m_2$ \cite{2007gerby}. We call this factor of $d$ the gluing factor.
\end{remark}

\begin{remark}
We denote the space of genus $g$ admissible cyclic degree three covers with $n$ points with monodromy $\omega$, $m$ points with monodromy $\overline{\omega}$, and $l$ unramified marked points as $Adm_{g(n|m|l)}$. We denote by $D_i^j$ the symmetrized boundary divisor in $\overline{\mathcal{M}}_{0,T}$ whose pullback via $\pi$ has $i \; \omega$ points and $j \; \overline{\omega}$ points on one component.
\end{remark}

\subsection{The Hodge Bundle}
The Hodge bundle $\mathbb{E}^g$ is a complex rank $g$ vector bundle over $Adm_g$, where the fiber of a curve is the vector space of one-forms. We denote $\lambda_i \defas c_i(\mathbb{E}^g)$. Along with the properties of Chern classes given earlier, the Chern classes of the Hodge bundle also satisfy Mumford's relation, which states that
$$(1+\lambda_1+\lambda_2+\lambda_g)(1-\lambda_1+\lambda_2 - \hdots \pm\lambda_g)=1$$ \cite{mumford}.
Another important concept which is important for our result is the projection formula.
Given the branch morphism $\pi: Adm_g \rightarrow \Mbar$, the projection formula states that for any curve $C$ in $\Mbar$,
  $$ \int_{C} \pi_*( \lambda_1) \; = \; \int_{\pi_*(C)} \hspace{-1mm} \lambda_1.$$
  
Let $D_i^j$, be a boundary divisor in $Adm_{g(n|m|l)}$. Then, $D_i^j \cong Adm_{g_1(n_1|m_1|l_1)} \times Adm_{g_2(n_2|m_2|l_2)}$. Then, we consider the Hodge bundle restricted to $D_i^j$.

Case 1, the monodromies at the the marked points connected at the node are unramified:
$$\mathbb{E}^g|_{D_i^j} \cong \mathbb{E}^{g_1} \oplus \mathbb{E}^{g_2}$$
Case 1, the monodromies at the the marked points connected at the node are ramified:
$$\mathbb{E}^g|_{D_i^j} \cong \mathbb{E}^{g_1} \oplus \mathbb{E}^{g_2} \oplus \mathcal{O}^2$$
where $\mathcal{O}$ is the trivial line bundle of $\C$.

In either case, the following relation between Chern classes of the Hodge bundle and Chern classes of the Hodge bundle restricted to a divisor can be derived.  Let $\lambda_1^L, \lambda_1^R$ denote $\lambda_1$ restricted to the left and right component of the divisor $D$ respectively.  Let the notation for the fundamental class be identical.  Then,

$$1+\lambda_1+\lambda_2+\cdots+\lambda_g = (1+\lambda_1^L+\lambda_2^L+\cdots+\lambda_g^L)(1+\lambda_1^R+\lambda_2^R+\cdots\lambda_g^R)$$

This is used to compute $\lambda_1$ of the Hodge bundle restricted to a boundary divisor. Namely, 

$$\lambda_1 = [1]^L\lambda_1^R + [1]^R\lambda_1^L.$$ This result generalizes to the Hodge Bundle restricted to boundary strata of higher codimension. In this case, $\lambda_1$ be equal to the sum of $\lambda_1$ restricted to each component of the boundary stratum (multiplied by a fundamental class on every other component).

\section{The First Chern Class of the Hodge Bundle over Spaces of $\faktor{\Z}{3\Z}$ Covers}
\begin{remark}
Let $\mathcal{C}$ be a boundary curve in $Adm_{n|m|l}$. Then, $\mathcal{C} \cong Adm_{n_1|m_1|l_1} \times ... \times Adm_{n_k|m_k|l_k}$ where $k = dim(Adm_{n|m|l}) - 1$. Since $\mathcal{C}$ is a curve, there is a single component, $\mathcal{C}_4$ for which $n_i + m_i + l_i = 4$ (for all other components, this sum equals three). This component is represented by a four-valent vertex in the graph dual to $\mathcal{C}$. The graph dual to $\mathcal{C}$ is shown in Figure \ref{markedgraph}. Let $\rho: \mathcal{C} \rightarrow \mathcal{C}_4$ be the projection from $\mathcal{C}$ onto this four-valent component.
\begin{figure}[h]
\begin{center}
\begin{tikzpicture}
\draw[very thick] (0,0) circle (0.3); \node at (0,0) {$D$};\node at (.7,0.25){$M_1$};
\draw[very thick] (.3,0) -- (1.2,0);
\fill (1.2,0) circle(0.1);

\draw[very thick] (1.2,1.2) circle (.3); \node at (1.2, 1.2){$A$}; \node at (1.45,0.5){$M_2$};
\draw[very thick] (1.2,0) -- (1.2,0.9);
\draw[very thick] (1.2,0) -- (1.2, -0.9); 
\draw[very thick] (1.2, -1.2) circle (.3); \node at (1.2,-1.2){$C$};\node at (.9,-.5){$M_4$};
\draw[very thick] (1.2,0) -- (2.1, 0);
\draw[very thick] (2.4,0) circle (.3); \node at (2.4,0){$B$}; \node at (1.7,-.25){$M_3$};
\end{tikzpicture}
\end{center}
\caption{Graph Dual to a Stratum in $Adm_{n|m|l}, A, B, C, D$ represent trivalent trees or single legs}
\label{markedgraph}
\end{figure}
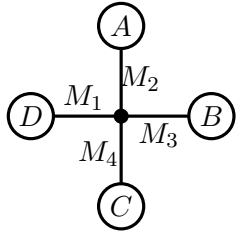

Denote by $M_i$ the monodromies of the marked points of $\mathcal{C}_4$. For simplicity in the proof of the main theorem, we will refer to these monodromies as being at the edge on the dual graph corresponding to the node connecting the marked point on the four-valent component to either $A, B, C,$ or $D$; for example, in Figure \ref{markedgraph} we would say the monodromy $M_2$ is the monodromy at the edge connecting $A$ to the four-valent component.
\end{remark}

\begin{lem} Every 1 dimensional boundary stratum in $\Mbar$ pulls back via the branch morphism to one of the four families in Figure \ref{families}.
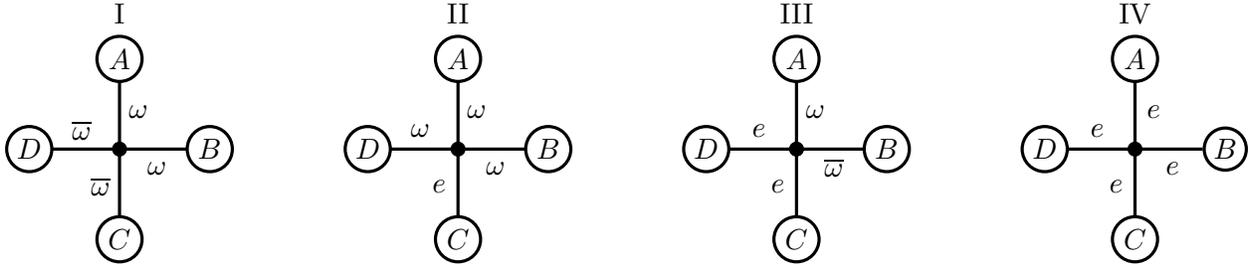
\begin{figure}
\begin{tikzpicture}
\draw[very thick] (0,0) circle (0.3); \node at (0,0) {$D$};\node at (.7,0.25){$\overline{\omega}$};
\draw[very thick] (.3,0) -- (1.2,0);
\fill (1.2,0) circle(0.1);

\draw[very thick] (1.2,1.2) circle (.3); \node at (1.2, 1.2){$A$}; \node at (1.45,0.5){$\omega$}; \node at (1.2,1.8){I};
\draw[very thick] (1.2,0) -- (1.2,0.9);
\draw[very thick] (1.2,0) -- (1.2, -0.9); 
\draw[very thick] (1.2, -1.2) circle (.3); \node at (1.2,-1.2){$C$};\node at (.95,-.5){$\overline{\omega}$};
\draw[very thick] (1.2,0) -- (2.1, 0);
\draw[very thick] (2.4,0) circle (.3); \node at (2.4,0){$B$}; \node at (1.7,-.25){$\omega$};

\draw[very thick] (4.5,0) circle (0.3); \node at (4.5,0) {$D$};\node at (5.2,0.25){$\omega$};
\draw[very thick] (4.8,0) -- (5.7,0);
\fill(5.7,0) circle(.1);

\draw[very thick] (5.7,1.2) circle (.3); \node at (5.7, 1.2){$A$}; \node at (5.95,0.5){$\omega$}; \node at (5.7,1.8){II};
\draw[very thick] (5.7,0) -- (5.7,0.9);
\draw[very thick] (5.7,0) -- (5.7, -0.9); 
\draw[very thick] (5.7, -1.2) circle (.3); \node at (5.7,-1.2){$C$};\node at (5.45,-.5){$e$};
\draw[very thick] (5.7,0) -- (6.6, 0);
\draw[very thick] (6.9,0) circle (.3); \node at (6.9,0){$B$}; \node at (6.2,-.25){$\omega$};

\draw[very thick] (9,0) circle (0.3); \node at (9,0) {$D$};\node at (9.7,0.25){$e$};
\draw[very thick] (9.3,0) -- (10.2,0);
\fill(10.2,0) circle(.1);

\draw[very thick] (10.2,1.2) circle (.3); \node at (10.2, 1.2){$A$}; \node at (10.45,0.5){$\omega$}; \node at (10.2,1.8){III};
\draw[very thick] (10.2,0) -- (10.2,0.9);
\draw[very thick] (10.2,0) -- (10.2, -0.9); 
\draw[very thick] (10.2, -1.2) circle (.3); \node at (10.2,-1.2){$C$};\node at (9.95,-.5){$e$};
\draw[very thick] (10.2,0) -- (11.1, 0);
\draw[very thick] (11.4,0) circle (.3); \node at (11.4,0){$B$}; \node at (10.7,-.25){$\overline{\omega}$};

\draw[very thick] (13.5,0) circle (0.3); \node at (13.5,0) {$D$};\node at (14.2,0.25){$e$};
\draw[very thick] (13.8,0) -- (14.7,0);
\fill(14.7,0) circle(.1);

\draw[very thick] (14.7,1.2) circle (.3); \node at (14.7, 1.2){$A$}; \node at (14.95,0.5){$e$}; \node at (14.7,1.8){IV};
\draw[very thick] (14.7,0) -- (14.7,0.9);
\draw[very thick] (14.7,0) -- (14.7, -0.9); 
\draw[very thick] (14.7, -1.2) circle (.3); \node at (14.7,-1.2){$C$};\node at (14.45,-.5){$e$};
\draw[very thick] (14.7,0) -- (15.6, 0);
\draw[very thick] (15.9,0) circle (.28); \node at (15.9,0){$B$}; \node at (15.2,-.25){$e$};
\end{tikzpicture}
\caption{Pullbacks of boundary strata in $\Mbar$ via the branch morphism}
\label{families}
\end{figure}
\end{lem}
\begin{proof}
The stratum in $\Mbar$ corresponding to a trivalent tree has the highest possible codimension and so is of dimension zero. Making a single node 4-valent decreases the codimension by one and thus creates a curve. Pulling this curve back via the forgetful morphism introduces the monodromy of each marked point and thus gives us different cases, depending on the monodromy of each marked point. The monodromy of each branch point on the four-valent component is determined by the congruence of the number of $\omega$ points minus the number of $\overline{\omega}$ points on the trivalent tree or single leg attached to the four-valent component at that node. That is, let $i_A, i_B, i_C, i_D$ denote the number of $\omega$ points on $A, B, C, D$, respectively and $j_A, j_B, j_C, j_D$ the number of $\overline{\omega}$ points on $A, B, C, D$, respectively. Then, the edge on the four-valent vertex connected to $A$ will have monodromy $e$ if $i_A - j_B \equiv 0 \; mod \; 3$, $\omega$ if $i_A - j_A \equiv 1 \; mod \; 3$ and $\overline{\omega}$ if $i_A - j_A \equiv -1 \; mod \; 3$, and likewise for $B, C, D$. Since the product of the monodromies at all branch points must equal $e$, the sum of all the congruences $i_A - j_A + i_B - j_B + i_C - j_C + i_D - j_D$ must be congruent to $0 \; mod \; 3$. There are only four possible ways to add four terms modulo 3 to be congruent to 0 modulo 3, which correspond to the four families above. Here we consider the case where there are three $\omega$ points and and identity at the nodes and the cases with three $\overline{\omega}$ points and an identity as the same family (Family II).
\end{proof}

\begin{lem}
$\lambda_1^{n+m-2+k} = 0$ over $Adm_{n|m|l}$ for any $k \in \Z_{\geq 0}$
\end{lem}
\begin{proof}
Consider the forgetful morphism $f: Adm_{n|m|l} \rightarrow Adm_{n|m}$, which forgets each unramified marked point. Then, 
$$[\lambda_1]_{Adm_{n|m|l}} = f^*([\lambda_1]_{Adm_{n|m}})$$
$$\Rightarrow ([\lambda_1]_{Adm_{n|m|l}})^{n+m-2} = f^*([\lambda_1]_{Adm_{n|m}})^{n+m-2})$$
Since the dimension of $Adm_{n|m} = n + m - 3$ and the codimension of $([\lambda_1]_{Adm_{n|m}})^{n + m - 2 + k} = n + m - 2 + k > n + m - 2$ for any $k \in \Z_{\geq 0}$,
$$([\lambda_1]_{Adm_{n|m}})^{n+m-2+k} = 0$$
$$\Rightarrow ([\lambda_1]_{Adm_{n|m|l}})^{n+m-2+k} = f^*(0) = 0$$ \end{proof}
\begin{lem}
Let $\Delta$ be a boundary stratum in $Adm_{n|m}$. Then, there exist a series of projections as shown in Figure \ref{lambdasplit}
\begin{figure}[h]
\[ \begin{tikzcd}
&\Delta \cong Adm_{(n_{1}|m_{1}|l_{1})} \times \hdots \times Adm_{(n_{k}|m_{k}|l_{k})} \arrow[swap]{dl}{\rho_1} \arrow{dr}{\rho_k} & \\
Adm_{(n_{1}|m_{1}|l_{1})} &\hdots& Adm_{(n_{k}|m_{k}|l_{k})}
\end{tikzcd} \]
\caption{$\lambda_1$ splitting over components of a boundary stratum}
\label{lambdasplit}
\end{figure}

Then, $[\lambda_1]_{Adm_{n|m}} = \sum \rho_i^*([\lambda_1]_{Adm_{n_i|m_i|l_i}})$. We say $\lambda_1$ splits over the components of $\Delta$.
\end{lem}
\begin{proof}
This comes by taking an iterated application of the process described in Figure \ref{split}.
\end{proof}
\label{rem1}

\begin{lem}
Let $\gamma$ be a boundary curve in $\Mbar$. Then, $\pi_*(\lambda_1) \cdot \gamma=0$ for curves which pull back to Families II, III and IV and $\pi_*(\lambda_1) \cdot \gamma = \frac{2}{9}$ for curves which pull back to Family I from Lemma 3.1. 
\end{lem}
\begin{proof}
By Lemma 3.3, $\pi^*(\gamma)$ is isomorphic to a product of admissible cover spaces. Since all but one of these has exactly three marked points, they are each $\frac{1}{3}$ the class of a point, since each space has three non-trivial automorphisms. However, for each edge, we must also multiply by the gluing factor of $3$. Denote the four-valent component as $\mathcal{C}_4$ and let $V$ be the set of tri-valent vertices in the graph dual to $\gamma$ and $E$ be the set of edges in the dual graph. Therefore,
\begin{align*}
\lambda_1 \cdot \pi^*(\gamma) &= \int_{\mathcal{C}_4} \lambda_1 \cdot \prod_{E}3 \cdot \prod_{V \neq \mathcal{C}_4}\frac{1}{3}\\
&= \int_{\mathcal{C}_4} \lambda_1 \text{ since $|E|$ = $|V|$ for a one-dimensional stratum}
\end{align*}
This integral is 0 if $\gamma$ pulls back to Families II, III, or IV by Lemma 3.2 and $\frac{2}{9}$ if $\gamma$ pulls back to Family I, as a result of a computation in \cite{2007gerby}.
\end{proof}

\begin{remark}
If $\pi$ is a degree $d$ map, $\pi^*\pi_*(\lambda_1) = d \lambda_1$.
\end{remark}

\begin{thm}
$\lambda_1 \text{ over } Adm_{(n|m)}$ can be expressed as $3\pi^*(\sum \alpha_i^jD_i^j)$ over all symmetrized boundary divisors $D_i^j$ of $\overline{\mathcal{M}}_{0,T}$
$$\text{where } \alpha_i^j = \begin{cases}
\frac{2(i + j)(T - i - j)}{27(T - 1)}&i - j \equiv 0 \; mod \; 3\\
\frac{2(i + j - 1)(T - i - j - 1)}{27(T - 1)}& i - j \equiv \pm 1 \; mod \; 3
\end{cases}$$
\end{thm}

\begin{proof}
Note: the coefficient in front of each symmetrized boundary divisor depends only on the sum of the $\omega$ and $\overline{\omega}$ points on one component. Therefore, we can denote $\alpha_i^j$ as $\alpha_{i+j}$, and if we let $t_1 = i + j$ denote the number of marked points on one component and $t_2 = T - i - j$ the number of marked points on the other component, we can express
$$\alpha_{t_{1}} = \begin{cases}
\frac{2t_1t_2}{27(T - 1)} & \text{monodromy at node: } e\\
\frac{2(t_1 - 1)(t_2 - 1)}{27(T - 1)} & \text{monodromy at node: } \omega, \overline{\omega}
\end{cases}
$$
We use this notation in the proof for simplicity.

By Remark 3.3, the statement of this theorem is equivalent to saying that $\pi_*\lambda_1 = \sum \alpha_i^jD_i^j$. Therefore, to prove this theorem, we show that for every one-dimensional boundary stratum $\gamma \in \overline{\mathcal{M}}_{0,T}$, 
\begin{equation}
    \gamma \cdot \pi_*\lambda_1 = \gamma \cdot \sum\alpha_i^jD_i^j
\end{equation}
The dual graph of a one-dimensional boundary stratum contains a single four-valent vertex with each edge connected to either a trivalent tree or to a single leg (that is, it is a half-edge representing a marked point). We denote this as in Figure \ref{curve}.\\
\begin{figure}
\begin{tikzpicture}
\draw[very thin] (0,0) circle (0);
\draw[very thick] (7.05,0) circle (0.3); \node at (7.05,0) {$D$};
\draw[very thick] (7.35,0) -- (8.25,0);
\fill(8.25,0) circle(.1);

\draw[very thick] (8.25,1.2) circle (.3); \node at (8.25, 1.2){$A$};
\draw[very thick] (8.25,0) -- (8.25,0.9);
\draw[very thick] (8.25,0) -- (8.25, -0.9); 
\draw[very thick] (8.25, -1.2) circle (.3); \node at (8.25,-1.2){$C$};
\draw[very thick] (8.25,0) -- (9.15, 0);
\draw[very thick] (9.45,0) circle (.3); \node at (9.45,0){$B$};
\end{tikzpicture}
\caption{Arbitrary Curve in $\Mbar$}
\label{curve}
\end{figure}
Let $\gamma$ be a curve in $\Mbar$ and let $t_a, t_b, t_c, t_d$ be the total number of marked points in $A, B, C, D$, respectively, where $A, B, C, D$ are all either trivalent trees or single half-edges (if $t_A, t_B, t_C, t_D = 1$, respectively). For this proof, we calculate
\begin{align}
\sum \alpha_{i+j} \gamma \cdot D_i^j 
\end{align}
and show that $\gamma \cdot \pi_*(\lambda_1)$ is equal to (2). By Lemma 3.3, $\lambda_1$ over $\pi^*(\gamma)$ is determined uniquely by $\lambda_1$ on the four-valent component, so we only need to prove the four cases where $\gamma$ pulls back to each of the four Families described in Lemma 3.1. It is easy to see that the only non-zero intersections between $\gamma$ and the sum of symmetrized boundary divisors are the non-transverse intersections with $D_A, D_B, D_C, D_D$ and the transverse intersections with $D_{A \cup B}, D_{A \cup C}, D_{A \cup D}$, where $D_X$ represents the boundary divisor with all the points in $X$ on one component and all other points on the other component. For example, $D_{A \cup B}$ would be represented by the dual graph in Figure \ref{dab}.
\begin{figure}[h]
\begin{center}
\begin{tikzpicture}
  \draw[very thick] (-1,0) -- (1,0);
  \fill (1,0) circle(.1);
  \fill (-1,0) circle(.1);
  \draw[very thick] (1,0) -- (1.5,.75); \node at (1.5,.9){$y_1$};
  \draw[very thick] (1,0) -- (1.5,.5); \node at (1.8,.5){$y_2$};
  \node at (1.4,.12){$\vdots$};
  \draw[very thick] (1,0) -- (1.5,-.5); \node at (2,-.5){$y_{t_{1}-1}$};
  \draw[very thick] (1,0) -- (1.5,-.75); \node at (1.5,-.9){$y_{t_1}$};
  \draw[very thick] (-1,0) -- (-1.5,.75); \node at (-1.5,.9){$x_1$};
  \draw[very thick] (-1,0) -- (-1.5,.5); \node at (-1.8,.5){$x_2$};
  \node at (-1.4,.12){$\vdots$};
  \draw[very thick] (-1,0) -- (-1.5,-.5); \node at (-2,-.5){$x_{t_{2}-1}$};
  \draw[very thick] (-1,0) -- (-1.5,-.75); \node at (-1.5,-.9){$x_{t_2}$};
  \end{tikzpicture}\\
  \textit{where $x_i$ are the marked points in $A \cup B$ and $y_i$ are the marked points in $C \cup D$}
\end{center}
\caption{The graph dual to the divisor $D_{A \cup B}$}
\label{dab}
\end{figure}
The intersections with the divisors $D_{A \cup B}, D_{A \cup C}, D_{A \cup C}$ are supported on the zero-dimensional boundary divisor which adds a single edge separating the points on $A \cup B, A \cup C, \text{ or } A \cup D$ from all other points, respectively. For example, the intersection between $\gamma$ and the divisor $D_{A \cup B}$ is supported by the dual graph in Figure \ref{intersectdab}.
\begin{figure}
\begin{center}
\begin{tikzpicture}
  \fill (.5,0) circle(.1);
  \fill (-.5,0) circle(.1);
  \draw[very thick] (0,0) -- (.5,0);
  \draw[very thick] (0,0) -- (-.5,0);
  \draw[very thick] (.5,0) -- (1,.5);
  \draw[very thick] (-.5,0) -- (-1,.5);
  \draw[very thick] (.5,0) -- (1,-.5);
  \draw[very thick] (-.5,0) -- (-1,-.5);
  \draw[very thick] (-1.21,.71) circle(0.3); \node at (-1.21,.71){$A$};
  \draw[very thick] (-1.21,-.71) circle(0.3); \node at (-1.21,-.71){$B$};
  \draw[very thick] (1.21,.71) circle(0.3); \node at (1.21,.71){$C$};
  \draw[very thick] (1.21,-.71) circle(0.3); \node at (1.21,-.71){$D$};
\end{tikzpicture}
\end{center}
\caption{$\gamma \cdot D_{A \cup B}$}
\label{intersectdab}
\end{figure}
Since this stratum is the product of a stratum of codimension one and a stratum of codimension equal to the dimension of $\Mbar - 1$, it is a stratum of dimension 0. Therefore, $\gamma \cdot D_{A \cup B}$ is the class of a point and thus contributes $1$ to the overall sum in (2). The intersections with $D_A, D_B, D_C, D_D$ are non-transverse intersections. These each have a common node with $\gamma$, connecting the four-valent component to a trivalent component. Therefore, intersecting these two strata reduces to computing 
$$-\psi_{n_1} - \psi_{n_2}$$
where $n_1, n_2$ once glued together by pulling back via the respective projections form the node connecting the two components. Since the common node is between an $\overline{\mathcal{M}}_{0,3}$ and an $\overline{\mathcal{M}}_{0,4}$, the intersection is computed by 
$$-[\psi_{n_1}]_{\overline{\mathcal{M}}_{0,3}} - [\psi_{n_1}]_{\overline{\mathcal{M}}_{0,4}} = 0 - 1 = -1.$$ This means that each intersection of the curve $\gamma$ with a divisor $D_A, D_B, D_C, D_D$ contributes $-1$ to the overall sum in (2). Therefore, verifying (1) reduces to showing that
\begin{align}
\gamma \cdot \pi_*(\lambda_1) &= \gamma \cdot (\alpha_AD_A + \alpha_BD_B + \alpha_CD_C + \alpha_DD_D + \alpha_{A \cup B}D_{A \cup B} + \alpha_{A \cup C}D_{A \cup C} + \alpha_{A \cup D}D_{A \cup D})\nonumber \\
&= -\alpha_A - \alpha_B - \alpha_C - \alpha_D + \alpha_{A \cup B} + \alpha_{A \cup C} + \alpha_{A \cup D}
\end{align}
Note: if any of $A, B, C, D$ are single half-edges rather than trivalent trees, the divisor $D_{A}, D_B, D_C,$ or $D_D$ is not a valid boundary divisor given our compactification of $\M$ as its dual graph would not be stable. However, this problem resolves itself combinatorially as the coefficient in front of these divisors becomes zero since $1 - 0 \equiv 1 \; mod \; 3$ so $\alpha_X = (t_X - 1)(\hdots) = (1 - 1)(\hdots) = 0$ for whichever set $X$ is a single half-edge.

We now verify (3) for each of the four cases described in Figure 4.

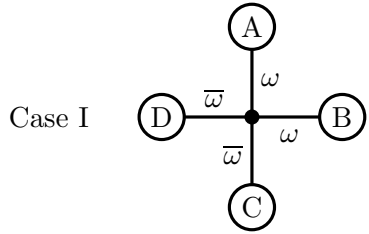
\begin{figure}[h]
\begin{tikzpicture}
\draw[very thin] (0,0) circle (0); \node at (5.55,0){Case I};
\draw[very thick] (7.05,0) circle (0.3); \node at (7.05,0) {D};\node at (7.75,0.25){$\overline{\omega}$};
\draw[very thick] (7.35,0) -- (8.25,0);
\fill (8.25,0) circle(.1);

\draw[very thick] (8.25,1.2) circle (.3); \node at (8.25, 1.2){A}; \node at (8.5,0.5){$\omega$};
\draw[very thick] (8.25,0) -- (8.25,0.9);
\draw[very thick] (8.25,0) -- (8.25, -0.9); 
\draw[very thick] (8.25, -1.2) circle (.3); \node at (8.25,-1.2){C};\node at (8,-.5){$\overline{\omega}$};
\draw[very thick] (8.25,0) -- (9.15, 0);
\draw[very thick] (9.45,0) circle (.3); \node at (9.45,0){B}; \node at (8.75,-.25){$\omega$};
\end{tikzpicture}
\end{figure}

First, we determine which of the two cases each coefficients $\alpha_A, \alpha_B, \alpha_C, \alpha_D, \alpha_{A \cup B}, \alpha_{A \cup C}, \alpha_{A \cup D}$ fall into. We do this by using the monodromy at each node to determine $i_X - j_X \; mod \; 3$ for $X$ any of the relevant sets. For example, the monodromy at the node of the four valent vertex connected to $A$ is of type $\omega$. Therefore, since the difference in $\omega$ and $\overline{\omega}$ points must be congruent to 0 $mod$ 3, $i_A + 1 - j_A \equiv 0 \; mod \; 3$. Therefore, $i_A - j_A \equiv -1 \; mod \; 3$. Doing this for each set gives us the following table:

\begin{tabular}{c|c c c c c c c}
Coefficient&$\alpha_A$&$\alpha_B$&$\alpha_C$&$\alpha_D$&$\alpha_{A \cup B}$&$\alpha_{A \cup C}$ & $\alpha_{A \cup D}$\\ \hline
Equivalence mod 3&-1&-1&1&1&1&0&0
\end{tabular}
$$\pi_*\lambda_1 \cdot C = \frac{2}{9}$$ by Lemma 3.4.
\begin{align*}
\Rightarrow \frac{2}{9} &= \frac{2}{27(T - 1)}(-(t_A - 1)(t_B + t_C + t_D - 1) - (t_B - 1)(t_A + t_C + t_D - 1)\\
&- (t_C - 1)(t_A + t_B + t_D - 1) - (t_D - 1)(t_A + t_B + t_C - 1) + (t_A + t_B - 1)(t_C + t_D - 1)\\
&+ (t_A + t_C)(t_B + t_D) + (t_A + t_D)(t_B + t_C))
\end{align*}
This equality is easily verifiable. Therefore, all curves of the form described above intersect correctly with $\pi_*\lambda_1$.

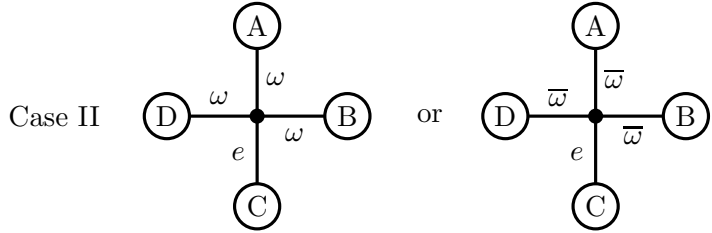
\begin{figure}[h]
\begin{tikzpicture}
\draw[very thin] (0,0) circle (0); \node at (8,0){or}; \node at (3,0){Case II};
\draw[very thick] (4.5,0) circle (0.3); \node at (4.5,0) {D};\node at (5.2,0.25){$\omega$};
\draw[very thick] (4.8,0) -- (5.7,0);
\fill (5.7,0) circle(.1);

\draw[very thick] (5.7,1.2) circle (.3); \node at (5.7, 1.2){A}; \node at (5.95,0.5){$\omega$};
\draw[very thick] (5.7,0) -- (5.7,0.9);
\draw[very thick] (5.7,0) -- (5.7, -0.9); 
\draw[very thick] (5.7, -1.2) circle (.3); \node at (5.7,-1.2){C};\node at (5.45,-.5){$e$};
\draw[very thick] (5.7,0) -- (6.6, 0);
\draw[very thick] (6.9,0) circle (.3); \node at (6.9,0){B}; \node at (6.2,-.25){$\omega$};
\text{or}
\draw[very thick] (9,0) circle (0.3); \node at (9,0) {D};\node at (9.7,0.25){$\overline{\omega}$};
\draw[very thick] (9.3,0) -- (10.2,0);
\fill (10.2,0) circle(.1);

\draw[very thick] (10.2,1.2) circle (.3); \node at (10.2, 1.2){A}; \node at (10.45,0.5){$\overline{\omega}$};
\draw[very thick] (10.2,0) -- (10.2,0.9);
\draw[very thick] (10.2,0) -- (10.2, -0.9); 
\draw[very thick] (10.2, -1.2) circle (.3); \node at (10.2,-1.2){C};\node at (9.95,-.5){$e$};
\draw[very thick] (10.2,0) -- (11.1, 0);
\draw[very thick] (11.4,0) circle (.3); \node at (11.4,0){B}; \node at (10.7,-.25){$\overline{\omega}$};
\end{tikzpicture}
\end{figure}

We prove the case that each monodromy is an $\omega$ point for this proof. The proof for 3 $\overline{\omega}$ points is the same, except the equivalencies mod 3 change between 1 and -1. This does not change the proof since being congruent to $\pm 1$ gives the same case for determining the $\alpha_k$, and it is a simple matter to verify that each coefficient be in the same case regardless of whether the points are $\omega$ or $\overline{\omega}$. Again, we begin by determining the appropriate equivalencies mod 3.

\begin{tabular}{c|c c c c c c c}
Coefficient&$\alpha_A$&$\alpha_B$&$\alpha_C$&$\alpha_D$&$\alpha_{A \cup B}$&$\alpha_{A \cup C}$ & $\alpha_{A \cup D}$\\ \hline
Equivalence mod 3&-1&-1&0&-1&1&-1&1
\end{tabular}

$$\pi_*\lambda_1 \cdot C = 0$$
by Lemma 3.4.
\begin{align*}
\Rightarrow 0 &= \frac{2}{27(T - 1)}(-(t_A - 1)(t_B + t_C + t_D - 1) - (t_B - 1)(t_A + t_C + t_D - 1) - t_C(t_A + t_B + t_D)\\
&- (t_D - 1)(t_A + t_B + t_C - 1) + (t_A + t_B - 1)(t_C + t_D - 1) + (t_A + t_C - 1)(t_B + t_D - 1))\\
&+ (t_A + t_D - 1)(t_B + t_C - 1)
\end{align*}
Again, it is easy to verify this equality.

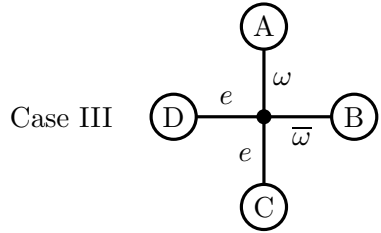
\begin{figure}[h]
\begin{tikzpicture}
\draw[very thin] (0,0) circle (0); \node at (5.55,0){Case III};
\draw[very thick] (7.05,0) circle (0.3); \node at (7.05,0) {D};\node at (7.75,0.25){$e$};
\draw[very thick] (7.35,0) -- (8.25,0);
\fill (8.25,0) circle(.1);

\draw[very thick] (8.25,1.2) circle (.3); \node at (8.25, 1.2){A}; \node at (8.5,0.5){$\omega$};
\draw[very thick] (8.25,0) -- (8.25,0.9);
\draw[very thick] (8.25,0) -- (8.25, -0.9); 
\draw[very thick] (8.25, -1.2) circle (.3); \node at (8.25,-1.2){C};\node at (8,-.5){$e$};
\draw[very thick] (8.25,0) -- (9.15, 0);
\draw[very thick] (9.45,0) circle (.3); \node at (9.45,0){B}; \node at (8.75,-.25){$\overline{\omega}$};
\end{tikzpicture}
\end{figure}

We first determine the equivalencies of the difference in $\omega, \overline{\omega}$ points for the relevant sets.

\begin{tabular}{c|c c c c c c c}
Coefficient&$\alpha_A$&$\alpha_B$&$\alpha_C$&$\alpha_D$&$\alpha_{A \cup B}$&$\alpha_{A \cup C}$ & $\alpha_{A \cup D}$\\ \hline
Equivalence mod 3&-1&1&0&0&0&-1&-1
\end{tabular}

$$\pi_*(\lambda_1) \cdot \gamma = 0$$
by Lemma 3.4.
\begin{align*}
\Rightarrow 0 &= \frac{2}{27(T - 1)}(-(t_A - 1)(t_B + t_C + t_D - 1) - (t_B - 1)(t_A + t_C + t_D - 1) - t_C(t_A + t_B + t_D)\\
&- t_D(t_A + t_B + t_C) + (t_A + t_B)(t_C + t_D) + (t_A + t_C - 1)(t_B + t_D - 1)\\
&+ (t_A + t_D - 1)(t_B + t_C - 1))
\end{align*}
Again, this equality is easily verifiable.

\begin{figure}[h]
\begin{tikzpicture}
\draw[very thin] (0,0) circle (0); \node at (5.55,0){Case IV};
\draw[very thick] (7.05,0) circle (0.3); \node at (7.05,0) {D};\node at (7.75,0.25){$e$};
\draw[very thick] (7.35,0) -- (8.25,0);
\fill (8.25,0) circle(.1);

\draw[very thick] (8.25,1.2) circle (.3); \node at (8.25, 1.2){A}; \node at (8.5,0.5){e};
\draw[very thick] (8.25,0) -- (8.25,0.9);
\draw[very thick] (8.25,0) -- (8.25, -0.9); 
\draw[very thick] (8.25, -1.2) circle (.3); \node at (8.25,-1.2){C};\node at (8,-.5){$e$};
\draw[very thick] (8.25,0) -- (9.15, 0);
\draw[very thick] (9.45,0) circle (.3); \node at (9.45,0){B}; \node at (8.75,-.25){$e$};
\end{tikzpicture}
\end{figure}

Similar to above, we first determine the equivalencies mod 3 and then verify the formula. Again, $\pi_*\lambda_1.C = 0$ by Lemma 3.4. In fact, since every node has monodromy $e$, every coefficient is in the case of equivalence to 0 mod 3.
\begin{align}
\Rightarrow 0 &= \frac{2}{27(T - 1)}(-t_A(t_B + t_C + t_D) - t_B(t_A + t_C + t_D) - t_C(t_A + t_B + t_D) \nonumber\\
&- t_D(t_A + t_B + t_C) + (t_A + t_B)(t_C + t_D) + (t_A + t_C)(t_B + t_D) \nonumber \\
&+ (t_A + t_D)(t_B + t_C))
\end{align}
Again, it is a simple matter to check that this equality holds.

Therefore, any one-dimensional boundary stratum in $\overline{\mathcal{M}}_{0,T}$ intersects with $\pi_*\lambda_1$ as predicted by our formula for the $\alpha_i^j$'s. Finally, in order to get $\lambda_1$ instead of its pushforward, we must pull the whole expression back by $\pi$. Since $\pi$ has degree $\frac{1}{3}$, by Remark 3.2,
\begin{align*}
&\pi^*\pi_*(\lambda_1) = \frac{1}{3}\lambda_1\\
\Rightarrow \lambda_1 &= 3\pi^*\pi_*(\lambda_1)\\
&= 3\pi^*(\sum \alpha_i^jD_i^j)
\end{align*}
\end{proof}
\section{The Second Chern Class of the Hodge Bundle on Spaces of Degree Two Admissible Covers.}
In this section we study the second Chern class of the Hodge bundle over spaces of degree two cyclic admissible covers. We first set up notation.

In this section, let let $Adm_g$ denote the space of admissible degree two covers of genus g.  In some places it will be more efficient to identify the space of admissible covers by number of branch points and not genus.  We will call the space of admissible covers with j branch points $Adm_j$. Finally, we will also work with spaces of admissible covers with $i$ marked branch points and $1$ marked identity point, we denote this space $Adm_{i,1}$

Our notation of boundary divisors in the space of admissible covers will vary as well in places.  We let $\Delta_i$ denote the boundary divisor with $i$ branch points on the left component. Some proofs necessitate that we denote a divisor by its number of branch points on the left component and by the space it resides in.  In these cases we let $\Delta_i^n$ denote that the divisor is in the space of admissible covers with $n$ marked points.  If an identity point is also present, the divisor will be denoted $\Delta_{i,1}$

\begin{lem}\cite{champs}
Let $\lambda_1$ be the first Chern class of the Hodge bundle over $Adm_{g}$. Furthermore, let $\Delta_i$ denote the codimenesion 1 symmeterized stratum parameterizing nodal covers with $i$ points on the left twig.  Let $N$ denote the total number of branch points on each divisor. Then, $\lambda_1$ can be expressed as $\sum \alpha_i\Delta_i$ over all symmetrized boundary divisors $\Delta_i$ of $Adm_g$
\begin{equation}\text{where } \alpha_i = \begin{cases}
\frac{i(N-i)}{8(N - 1)}&i \equiv 0 \; mod \; 2\\
\frac{(i - 1)(N-i-1)}{8(N - 1)}& i \equiv 1 \; mod \; 2
\end{cases}\end{equation}
\end{lem}
Our main goal in this section is to extend these results to the second Chern class.  The first step is a consequence of a theorem of Keel which states that $\lambda_2$ can be expressed as a linear combination of codim 2 boundary strata in $Adm_g$.
\begin{lem}\cite{keel}
Let $\lambda_2$ denote the second Chern class of the Hodge bundle over $Adm_{g}$. Let $\Delta_{i_1i_2i_3}$ denote the symmeterized stratum with $i_1$ points on the left component, $i_2$ points in the center, and $i_3$ components on the right.  Then, there exists rational coefficients $\alpha_{i_1,i_2,i_3}$ such that 
$$\lambda_2=\sum \alpha_{i_1,i_2,i_3} \Delta_{i_1i_2i_3}$$
\end{lem}
Our theorem in this section determines these coefficients.  To this end, we will use the following facts.  
\begin{lem}
 $$\lambda_2=\frac{1}{2}\lambda_1^2= \frac{1}{2} \sum \alpha_i^n \lambda_1 \Delta_i^n$$
 \begin{proof}
 The first equality follows from Mumford's relations\cite{mumford} the second follows from Lemma $4.1$
 \end{proof}
\end{lem}
Following notation from Section 2.6 we have
\begin{lem}

$$\alpha_i^n\lambda_1 \Delta_i^n = \alpha_i^n([1]^L[\lambda_1]^R\oplus[1]^R[\lambda_1^L])$$
$$= \begin{cases}
\alpha_i^{n}([1]|Adm_{i,1}[\lambda_1]|Adm_{n-i,1}\oplus[\lambda_1]|Adm_{i,1}[1]|Adm_{n-i,1}) \ \ \ i \equiv 0 \; mod \; 2\\
\alpha_i^{n}([1]|Adm_{i+1}[\lambda_1]|Adm_{n-i+1}\oplus[\lambda_1]|Adm_{i+1}[1]|Adm_{n-i+1}) \equiv 1 \; mod \; 2
\end{cases}$$  
\end{lem} 
\begin{lem}
$$[\lambda_1]|Adm_{n-i,1} = \sum \alpha_j^{n-i}\Delta_{j,1}^{n-i}$$
$$[\lambda_1]|Adm_{i,1} = \sum \alpha_j^{i}\Delta_{j,1}^{i}$$
$$[\lambda_1]|Adm_{n-i+1} = \sum \alpha_{j+1}^{n-i+1}\Delta_{j+1}^{n-i+1}$$
$$[\lambda_1]|Adm_{i+1} = \sum \alpha_{j+1}^{i+1}\Delta_{j+1}^{i+1}$$
\end{lem}
\begin{proof}
For the case with with no identity point simply apply lemma 4.1.  For the case with an identity point note that $[\lambda_1]|Adm_{i,1} = \sum \alpha_{j,1}^{i}\Delta_{j,1}^{i}$.  Now consider the forgetful morphism $\phi: Adm_{i,1} \rightarrow Adm_{i}$.  Then $\alpha_{j,1}^{i}\Delta_{j,1}^{i} = \phi^*(\alpha_j\Delta_{j})$ so that $\alpha_{j,1} = \alpha_{j}$
\end{proof}
The strategy of the proof for the following theorem is to fix a divisor $\Delta_{i_1,i_2,i_3}$ in the boundary expression for $\lambda_2$ and track when it appears in the expression of lemma 4.5.  This allows us write the coefficient $\alpha_{i_1,i_2i_3}$ in terms of the coefficients $\alpha_i$, the latter of which we have an explicit formula for.

\begin{thm}
Let $\Delta_{i_1,i_2,i_3}$ denote the codimension 2 symmeterized stratum in $Adm_g$ with $i_1$ branch points on the left component, $i_2$ branch points on the middle component, and $i_3$ branch points on the right component. Then $\lambda_2 = 2\sum \alpha_{i_1,i_2,i_3} \Delta_{i_1,i_2,i_3}$ where
$$\ \alpha_{i_1,i_2,i_3} = \begin{cases}
\frac{i_1i_2i_3(2i_1i_2+2i_1i_3+2i_2i_3-i_1-2i_2-i_3)}{32(i_1+i_2+i_3-1)(i_1+i_2-1)(i_2+i_3-1)}&i_1,i_2,i_3 \equiv 0 \; mod \; 2\\
 \frac{(i_1-1)(i_2)(i_3-1)((i_2+i_3-1)(i_1+i_2)+(i_1+i_2-1)(i_2+i_3))}{32(i_1+i_2+i_3-1)(i_1+i_2)(i_2+i_3)}&i_1, i_3 \equiv 1,\ \ i_2 \equiv 0 \; mod \; 2 \\
 \frac{(i_1-1)(i_2+i_3-1)(i_2+1)(i_3)(i_1+i_2-1)+(i_3)(i_1+i_2)(i_2-1)(i_1-1)(i_2+i_3)}{32(i_1+i_2+i_3-1)(i_2+i_3)(i_1+i_2-1)}&i_1, i_2 \equiv 1, \ \ i_3 \equiv 0 \; mod \; 2 \\
\end{cases}$$  
\end{thm}
\begin{proof}
\textbf{Case 1:} $i_1,i_2,i_3$ even

Let $i_1+i_2+i_3=n$.
 Since $i_1,i_2,i_3$ are even, all nodes of $\Delta_{i_1,i_2,i_3}$ are unramified. Thus the divisor $\Delta_{i_1,i_2,i_3}$ only appears in the case $i$ even of lemma 4.5.  Using this, we have that

\begin{equation}
\sum \alpha_{i_1,i_2,i_3} \Delta_{i_1,i_2,i_3} = \sum \alpha_i^{n} ([1]_{Adm_{i,1}} \times \sum \alpha_j^{n-i}\Delta_{j,1}^{n-i}  \oplus [1]_{Adm_{n-i,1}} \times \sum \alpha_j^{i}\Delta_{j,1}^{i})    
\end{equation}

Now we fix a given $\Delta_{i_1i_2i_3}$ and compute its' coefficient in (6).  The expression $[1]_{Adm_{i,1}} \times \Delta_{j,1}^{n-i} $ gives the divisor $\Delta_{i_1i_2i_3}$ when $i = i_1$ and $j = i_2$, or $i=i_3$ and $j=i_2$.  The expression $[1]_{Adm_{n-i,1}} \times \Delta_{j,1}^{i}$  gives the divisor $\Delta_{i_1i_2i_3}$ when $i = i_2 +i_3 $ and $j=i_2$, or when $i = i_1 + i_2$ and $j=i_2$.  Combining the above lemmas we have the following for the coefficient of a fixed $\Delta_{i_1i_2i_3}$  

$$\alpha_{i_1}^n\alpha_{i_2}^{i_2+i_3} \ + \ \alpha_{i_3}^n\alpha_{i_2}^{i_1+i_2} \ + \ \alpha_{i_2+i_3}^n\alpha_{i_2}^{i_2+i_3} \ + \ \alpha_{i_1+i_2}^n\alpha_{i_2}^{i_1+i_2} = \alpha_{i_1i_2i_3}$$  Using the expressions $\alpha_{i_1}^n = \alpha_{i_2+i_3}^n$ and $\alpha_{i_3}^n = \alpha_{i_1+i_2}^n$ we have the expression

$$\alpha_{i_1}^{n}\alpha_{i_2}^{i_2+i_3} \ + \ \alpha_{i_3}^{n}\alpha_{i_2}^{i_1+i_2}  = \frac{1}{2}\alpha_{i_1i_2i_3}$$ 
which holds when $i_1,i_2,i_3$ are even.
 Plugging in the formula for each $\alpha$ from (4) into the above we get the coefficient $2(\frac{(i_1+i_2)i_3}{8(i_1+i_2+i_3-1)}\frac{i_1i_2}{8(i_1+i_2-1)} \ + \ \frac{i_1(i_2+i_3)}{8(i_1+i_2+i_3-1)}\frac{i_2i_3}{8(i_2+i_3-1)})$ for the divisor $\Delta_{i_1,i_2,i_3}$ Upon simplification we arrive at the expression $$\frac{i_1i_2i_3(2i_1i_2+2i_1i_3+2i_2i_3-i_1-2i_2-i_3)}{32(i_1+i_2+i_3-1)(i_1+i_2-1)(i_2+i_3-1)}$$
\textbf{Case 2:} $i_1,i_3$ odd, $i_2$ even

 Since $i_1, i_3$ are odd, all nodes of $\Delta_{i_1,i_2,i_3}$ are ramified.  Thus, the divisor $\Delta_{i_1,i_2,i_3}$ only appears in the case $i$ odd of lemma 4.5.  Given this we have  
\begin{equation}
\sum \alpha_{i_1,i_2,i_3} \Delta_{i_1,i_2,i_3} = \sum \alpha_i^{n} ([1]_{Adm_{i+1}} \times \sum \alpha_{j+1}^{n-i+1}\Delta_{j+1}^{n-i+1}  \oplus [1]_{Adm_{n-i+1}} \times \sum \alpha_{j+1}^{i+1}\Delta_{j+1}^{i+1})   
\end{equation}
Again we fix a given $\Delta_{i_1i_2i_3}$ and compute its' coefficient in (7).  The expression $[1]_{Adm_{i+1}} \times \Delta_{j+1}^{n-i+1} $ gives the divisor $\Delta_{i_1,i_2,i_3}$ when $i=i_1$ and $j=i_2$ or when $i=i_3$ and $j=i_2$.  The expression $[1]_{Adm_{n-i+1}} \times \Delta_{j+1}^{i+1}$ gives the divisor $\Delta_{i_1,i_2,i_3}$ when $i=i_2+i_3$ and $j=i_2+1$ or when $i=i_1+i_2$ and $j=i_2$.  Combining these gives the following coefficient for a fixed $\Delta_{i_1,i_2,i_3}$
$$\alpha_{i_1}^n\alpha_{i_2+1}^{i_2+i_3+1} \ + \ \alpha_{i_3}^n\alpha_{i_2+1}^{i_1+i_2+1} \ + \ \alpha_{i_2+i_3}^n\alpha_{i_2+1}^{i_2+i_3+1} \ + \ \alpha_{i_1+i_2}^n\alpha_{i_2+1}^{i_1+i_2+1} = \alpha_{i_1i_2i_3} $$

which can be regrouped into
$$\alpha_{i_1}^{n}\alpha_{i_2+1}^{i_2+i_3+1} \ + \ \alpha_{i_3}^{n}\alpha_{i_2+1}^{i_1+i_2+1}  = \frac{1}{2}\alpha_{i_1i_2i_3}$$
Here, $i_1$ is odd, $i_2+1$ is odd and $i_3$ is odd.  Using this we plug in the expressions from Lemma \textbf{4.1}.  
$$\alpha_{i_1,i_2,i_3} = 2(\frac{(i_1-1)(n-i_1-1)}{8(n-1)}\frac{(i_2)(i_3-1)}{8(i_2+i_3)} + \frac{(i_3-1)(n-i_3-1)}{8(n-1)}\frac{(i_2)(i_1-1)}{8(i_1+i_2)})$$ or
$$\alpha_{i_1,i_2,i_3} = \frac{(i_1-1)(i_2)(i_3-1)((i_2+i_3-1)(i_1+i_2)+(i_1+i_2-1)(i_2+i_3))}{32(i_1+i_2+i_3-1)(i_1+i_2)(i_2+i_3)}$$

\textbf{Case 3:} $i_1$ odd $i_2$ odd $i_3$ even

In this case the left node is ramified and the other is unramified.  Thus the divisor $\Delta_{i_1,i_2,i_3}$ appears in lemma 4.5 as follows.

\begin{multline}
\sum \alpha_{i_1,i_2,i_3} \Delta_{i_1,i_2,i_3} = \sum \alpha_i^{n} ([1]_{Adm_{i+1}} \times \sum \alpha_{j+1}^{n-i+1}\Delta_{j+1}^{n-i+1} \oplus [1]_{Adm_{n-i+1}} \times \sum \alpha_{j+1}^{i+1}\Delta_{j+1}^{i+1}) \\
 + \sum \alpha_{t}^{n}([1]_{Adm_{n-t,1}} \times \sum \alpha_j^{t}\Delta_{j,1}^{t}  \oplus [1]_{Adm_{n-t,1}} \times \sum \alpha_j^{t}\Delta_{j,1}^{t})  )
 \end{multline}

Now we fix the divisor $\Delta_{i_1,i_2,i_3}$ and see when it appears in (8).  The divisor is given when $i=i_1$ and $j=i_2$, $n-i = i_1$ and $j=i_2$ or when $t=i_1+i_2$ and $j=i_2$, $n-t = i_3$ and $j=i_2$.  This gives the following equation for the fixed coefficient $\alpha_{i_1,i_2,i_3}$.

$$\alpha_{i_1,i_2,i_3} = 2(\alpha_{i_1}^{n}\alpha_{i_2+1}^{i_2+i_3+1} + \alpha_{i_3}^{n}\alpha_{i_2}^{i_1+i_2})$$

Here, $i_1$ is odd, $i_2+1$ is even, $i_3$ is even and $i_2$ is odd.  Using this and the formulas from Lemma 4.1 we have $$\alpha_{i_1,i_2,i_3} = 2(\frac{(i_1-1)(n-i_1-1)}{8(n-1)}\frac{(i_2+1)(i_3)}{8(i_2+i_3)} + \frac{(i_3)(n-i_3)}{8(n-1)}\frac{(i_2-1)(i_1-1)}{8(i_1+i_2-1)})$$ or
$$\alpha_{i_1,i_2,i_3} = \frac{(i_1-1)(i_2+i_3-1)(i_2+1)(i_3)(i_1+i_2-1)+(i_3)(i_1+i_2)(i_2-1)(i_1-1)(i_2+i_3)}{32(i_1+i_2+i_3-1)(i_2+i_3)(i_1+i_2-1)} $$

\end{proof}
\section{Hodge Integrals}
In this section we use a technique that is similar to the one used in \cite{pete}.
\begin{thm}
The family of Hodge integrals $\int \lambda_{1}^{n+m-3}$ can be computed using the recursive formula 
$$\int_{Adm_{g(n|m)}} \hspace{-1.3cm} \lambda_{1}^{n+m-3} = 3 \sum_{i=0}^{n} \sum_{j=0}^{m} \frac{2(i+j-1)(T-i-j-1)}{9(T-1)} {n+m-3 \choose i+j-2}{n \choose i}{m \choose j}\int_{Adm_{(i+1|j)}} \hspace{-1.3cm} \lambda_{1}^{i+j-2}\int_{Adm_{(n-i|m-j+1)}} \hspace{-2.3cm} \lambda_{1}^{n+m-i-j-2} $$
\end{thm}
\begin{proof}
From Theorem 1.1, we can replace one $\lambda_1$ with its boundary expression, so the Hodge integral becomes 
\begin{equation}\int_{Adm_{g(n|m)}} \hspace{-0.5cm} \lambda_{1}^{n+m-3} = 3\sum_{i=0}^{n} \sum_{j=0}^{m} \alpha_i^j \int_{Adm_{g(n|m)}} \hspace{-1cm} \lambda_{1}^{n+m-4} \pi^*(D^{i}_{j})\end{equation} 
Now for a given $D_i^j$ we must evaluate $\int_{g(n|m)}\lambda_{1}^{n+m-4} \pi^*(D^{i}_{j})$. When we restrict $\lambda_1$ to the pull back of $D^i_j$ we use the Whitney formulas, $\lambda_{1}|D_i^j=\lambda_{1}^L|D_i^j+\lambda_{1}^R|D_i^j$. We raise this to the $n+m-4$ power using the Binomial expansion theorem and have $$\lambda_{1}^{n+m-4}|D^i_j = \sum_{k=0}^{n+m-4}{n+m-4 \choose k}(\lambda_{1}^{L}|D^i_j)^{k}(\lambda_{1}^R|D_i^j)^{n+m-4-k}$$ First, recall that in \textbf{Lemma 3.2} we proved that if $i-j$ is divisible by 3 the above expression is zero.  Furthermore, through a simple dimension count, we can see that the only non-zero term in this sum be when $k=i+j-2$. If we account for the fact that there are ${n \choose i}{m \choose j}$ irreducible boundary divisors in a given $D^i_j$, then we can see that   $$\int_{Adm_{g(n|m)}} \hspace{-1cm} \lambda_{1}^{n+m-4} \pi^*(D^{i}_{j}) = {n+m-3 \choose i+j-2}{n \choose i}{m \choose j}\int_{Adm_{g(i+1|j)}} \hspace{-1cm} \lambda_{1}^{i+j-2}\int_{Adm_{g(n-i|m-j+1)}} \hspace{-1cm} \lambda_{1}^{n+m-i-j-2}$$ Placing this and the expression for $\alpha_i^j$ when $i-j$ is not divisible by 3 into (7) completes the proof.  
\end{proof}
\begin{remark}
In order for the values of $i, j$ to define stable admissible cover spaces in the terms in Theorem 5.1, they must satisfy the following properties:
\begin{enumerate}
    \item $i - j \equiv 2 \; mod \; 3$
    \item $i + j \geq 2$
\end{enumerate}
It is also easy to note that from the combinatorial factor, the term when $i = n - 1, j = m$ must always equal zero. This, and the aforementioned conditions on $i, j$, greatly reduces the number of terms which may give a non-zero contribution to the sum given in Theorem 5.1. Also, notice we only allow $i - j \equiv 2 \; mod \; 3$. If we allowed $i - j \equiv 1 \; mod \; 3$, due to the symmetry of boundary divisors, we would get exactly the same sum and would thus need to multiply the sum by an extra factor of $\frac{1}{2}$.
\end{remark}

We use this formula to compute
$$\int_{Adm_{g(n|m)}} \hspace{-1cm} \lambda_{1}^{n+m-3}$$
for some values of $n, m$ in the following table.

\begin{center}
\begin{tabular}{|c|c|c|}
\hline
    $n$&$m$&$\int \lambda_1^{n+m-3}$\\ \hline
    3&0&$\frac{1}{3}$\\ \hline
    2&2&$\frac{2}{9}$\\ \hline 
    4&1&$\frac{4}{27}$\\ \hline
    6&0&$\frac{8}{27}$\\ \hline
    3&3&$\frac{128}{135}$\\ \hline
    5&2&$\frac{3392}{729}$\\ \hline
    4&4&$\frac{446923}{5103}$\\ \hline
\end{tabular}
\end{center}
Note that we do not include symmetric cases (where the values of $n$ and $m$ are switched) as the Hodge integrals are equal due to the symmetry of boundary divisors.

\bibliography{bibliography}
\bibliographystyle{plain}
\end{document}